\documentclass{amsart}
\usepackage{amsfonts}
\usepackage{amscd}
\usepackage{amssymb}
\usepackage{mathrsfs}
\theoremstyle{plain}
\newtheorem{thm}{Theorem}[section]

\newtheorem{lem}[thm]{Lemma}
\newtheorem{pro}[thm]{Proposition}
\theoremstyle{definition}
\newtheorem{ex}[thm]{Example}
\newtheorem{rmk}[thm]{Remark}
\newtheorem{defn}[thm]{Definition}
\newtheorem{assumption}[thm]{Assumption}
\usepackage{amsmath}
\usepackage{amssymb}
\usepackage{amsthm}
\usepackage{latexsym}
\usepackage[T1]{fontenc}
\usepackage[all]{xy}

\DeclareMathOperator{\fchar}{char}
\DeclareMathOperator{\End}{End}

\DeclareMathOperator{\Ker}{Ker}

\title[Commuting elements in the $q$-deformed Heisenberg algebra]{Algebraic curves for commuting elements in
the $q$-deformed Heisenberg algebra}

\author{Marcel de Jeu}
\address{Marcel de Jeu, Mathematical Institute,
Leiden University, P.O. Box 9512, 2300 RA Leiden,
The Netherlands}
\email{mdejeu@math.leidenuniv.nl}

\author{Christian Svensson}
\address{Christian Svensson, Mathematical Institute, Leiden University,
P.O. Box 9512, 2300 RA Leiden, The Netherlands,
and Centre for Mathematical Sciences, Lund
University, Box 118, SE-221 00 Lund, Sweden}
\email{chriss@math.leidenuniv.nl}

\author{Sergei Silvestrov}
\address{Sergei Silvestrov, Centre for Mathematical Sciences, Lund
University, Box 118, SE-221 00 Lund, Sweden}
\email{Sergei.Silvestrov@math.lth.se}

\begin{document}

\begin{abstract}
In this paper we extend the eliminant construction of Burchnall and Chaundy for
commuting differential operators in the Heisenberg algebra to the $q$-deformed
Heisenberg algebra and show that it again provides annihilating curves for
commuting elements, provided $q$ satisfies a natural condition. As a side
result we obtain estimates on the dimensions of the eigenspaces of elements of
this algebra in its faithful module of Laurent series.
\end{abstract}



\maketitle

\section{Introduction}\label{sec:introduction}
In the literature on algebraic dependence of commuting elements in the
Heisenberg algebra -- a result which is relevant for the algebro-geometric
method of solving certain non-linear partial differential equations -- one can
find several different proofs of this fact, each with its own advantages. The
first proof utilizes analytical methods and was found by Burchnall and Chaundy
\cite{Burchnall&Chaundy1} in the 1920's. It is basically their approach which
was rediscovered later and applied in the context of non-linear differential
and difference equations (see for example
\cite{Krichever,Mumford,MoerbekeMumford}, and for further references the book
\cite{HSbook}). Another and more algebraic method of proof for differential
operators was suggested by Amitsur \cite{Amitsur} in the 1950's, and in the
late 1990's a more algorithmic combinatorial method of proof was found
\cite{HDiscMath,HSbook}. One of the motivating problems for these developments
was to describe, as detailed as possible, commuting differential operators and
their properties
\cite{Burchnall&Chaundy1,Burchnall&Chaundy2,Burchnall&Chaundy3}. Clearly the
result of Burchnall and Chaundy \cite{Burchnall&Chaundy1}, stating that two
commuting differential operators in the Weyl algebra satisfy equations for
algebraic curves which can be explicitly calculated by the so-called eliminant
method, is then an important tool.

In 1994 Silvestrov, based on the existing literature and a series of trial
computations, conjectured, loosely speaking, that it should be true in a
considerably greater generality than the context of the Weyl algebra that two
commuting elements in an algebra lie on a curve, and, moreover, that the
eliminant construction of Burchnall and Chaundy should then produce such curves
in this wider context. We refer to \cite{SSJ} for more precise information on
this conjecture. The conjecture includes the $q$-deformed Heisenberg algebra
$H_K(q)$ of this paper, which is the associative algebra generated over a field
$K$ by two elements $A$ and $B$ subject to the relation $AB-qBA=1$. The case
$q=1$ and $K=\mathbb{R},\mathbb{C}$ yields the classical Weyl algebra for which
the result was known from the work by Burchnall and Chaundy.

There have been previous results supporting the conjecture for $H_K(q)$. In
\cite{HSbook} it was established under Assumption~\ref{assumptionq} below
(which essentially amounts to $q$ not being a root if unity) that
two commuting elements in $H_K(q)$ do in fact lie on a curve, using methods
which have since then been extended to more general algebras and rings
generalizing $q$-deformed Heisenberg algebras (generalized Weyl structures and
graded rings) in \cite{LHSSGWSergpotJAlg}. The proof in \cite{HSbook} is rather
different from the approach as followed by Burchnall and Chaundy. It is
constructive in the sense that it can be used effectively to compute algebraic
curves for any two given commuting elements, but it does not give additional a
priori information on, e.g., the coefficients of the curves or their degree.
The eliminant construction on the other hand \emph{does} provide such a priori
information (cf.\ Theorem~\ref{mainthm}), so that establishing the validity or
invalidity of this construction is a relevant issue. In \cite{LarssonSilv}, a
step in that direction was made by offering a number of examples all supporting
the conjecture that the eliminant construction should work for general
$H_K(q)$.

In this paper Silvestrov's conjecture for $H_K(q)$ is confirmed as
Theorem~\ref{mainthm} under Assumption~\ref{assumptionq}. For non-zero $q$ not
satisfying Assumption~\ref{assumptionq} it is known \cite{HDiscMath,HSbook}
that there are commuting elements in $H_K(q)$ which are algebraically
independent. Thus, as long as $q\neq 0$, the eliminant construction gives a
method to produce such curves precisely when the existence of such a method is
not excluded a priori, confirming the part of the conjecture concerning the
validity of the eliminant construction. The case $q = 0$ seems to be still
open.

In closing, let us remark that there are also results known about algebraic (in)dependence of commuting elements of the quantum plane (i.e., of the complex algebra generated by elements $A$ and $B$ subject to the condition $AB-qBA=0$) if $q$ is not a root of unity. It was proved by Artamonov and Cohn \cite{ArCo} that the commutant of an arbitrary non-constant element of the quantum plane is a commutative algebra of transcendence degree one, and later this result was sharpened by Makar-Limanov \cite{MaLi}, who gave a direct proof that this commutant is actually isomorphic to a subalgebra of $\mathbb C[X]$. We refer to \cite{MaLi} for a more detailed discussion of the commutant in the quantum plane case and the relevant literature.

\section{Basic notions and statement of the result}\label{sec:basicnotions}

In this section we introduce the basic notions and state our main result, Theorem~\ref{mainthm}. We
also explain the structure of the proof in the subsequent sections, which uses a
faithful module described in the current section. The section concludes with a remark on the difference with the original situation as considered by Burchnall and Chaundy and a description of the contents of the remaining sections.

Let $K$ be a field. If $q \in K$ then $H_K (q)$, the $q$-deformed Heisenberg
algebra over $K$, is the unital associative $K$-algebra which is generated by
two elements $A$ and $B$, subject to the $q$-commutation relation $AB - qBA =
I$. This algebra is sometimes also called the $q$-deformed Weyl algebra, or the
$q$-deformed Heisenberg-Weyl algebra, but we will follow the terminology in
\cite{HSbook}. We will prove that - under a condition on $q$ - for any
commuting $P, Q \in H_K (q)$ of order at least one (where ``order'' will be
defined below), there exist finitely many explicitly calculable polynomials
$p_i \in K[X, Y]$ such that $p_i (P, Q) =0$ for all $i$, and at least one of
the $p_i$ is non-zero. Thus $P$ and $Q$ lie on at least one algebraic curve. The number of polynomials $p_i$ depends not only on the order
of $P$ and $Q$, but also on their coefficients. The polynomials are
obtained by the analogue of the eliminant construction of Burchnall and Chaundy
\cite{Burchnall&Chaundy1,Burchnall&Chaundy2,Burchnall&Chaundy3} for the case
$q=1$ as mentioned in the introduction, and which we will now explain.

Define the $q$-integer $\{n\}_q$, for $n \in \mathbb{Z}$, by
\[\{n\}_q = \left\{ \begin{array}{lr}
\frac{q^n -1}{q-1} & q \neq 1;\\
 n & q = 1.
\end{array}\right.\]

In order for our method to work, and also for the eliminant construction to be well-defined to start with, we impose the following condition on $q$. It almost amounts to requiring that $q$ is not a root of unity, but one has to take the characteristic of $k$ into account.

\begin{assumption}\label{assumptionq}
Throughout this paper we assume that $q\neq 0$ and $\{n\}_q\neq 0$ if
$n\in\mathbb{Z}$ is non-zero.
\end{assumption}

\begin{rmk}\label{nonunroot}
The following are equivalent for $q \neq 0$:
\begin{enumerate}
\item for $n \in \mathbb{Z}$, $\{n\}_q =0$ if and only if $n=0$; \item for
$n_1, n_2 \in \mathbb{Z}$, $\{n_1\}_{q} = \{n_2\}_{q}$ if and only if $n_1 =
n_2$; \item $\left\{\begin{array}{lr}
q \textup{ is not a root of unity other than } 1, & \textup{if } \fchar k = 0;\\
q \textup{ is not a root of unity},  & \textup{if } \fchar k \neq 0.\\
\end{array}
\right.$
\end{enumerate}
Hence under our assumptions $K$ is infinite. Part (2) of this remark will prove
to be essential later on when we consider the dimension of eigenspaces.
\end{rmk}

Let $\mathcal{L}$ be the $K$-vector space of all
formal Laurent series in a single variable $t$
with coefficients in $K$. Define
\[M \left(\sum_{n= -\infty}^{\infty} a_n t^n\right) = \sum_{n = - \infty}^{\infty} a_n t^{n+1} = \sum_{n = - \infty}^{\infty} a_{n-1} t^n,\]
\[D_q \left(\sum_{n= -\infty}^{\infty} a_n t^n\right) = \sum_{n = - \infty}^{\infty} a_n \{n\}_q t^{n-1} = \sum_{n = - \infty}^{\infty} a_{n+1}
\{n+1\}_{q} t^{n}.\] Alternatively, one could
introduce $\mathcal{L}$ as the vector space of
all functions from $\mathbb{Z}$ to $K$ and let
$M$ act as the right shift and $D_q$ as a
weighted left shift, but the Laurent series model
is more appealing.

The algebra $H_K (q)$ has $\{I, A, A^2, \ldots\}$
as a free basis in its natural structure as a
left $K[X]$-module where $X$ acts as left multiplication with $B$.
If an arbitrary non-zero element $P$ of $H_K (q)$
is then written as
\[P = \sum_{j=0}^m p_j (B) A^j,\,\, p_m \neq 0,\]
for uniquely determined $p_j \in K[X]$ and $m
\geq 0$, then the integer $m$ is called the
\emph{order} of $P$ (or the degree of $P$ with respect to
A) \cite{HSbook}.

By sending $A$ to $D_q$ and $B$ to $M$, $\mathcal{L}$ becomes a faithful $H_K
(q)$-module, as is easily seen \cite{HSbook}. We will identify $H_K (q)$ with
its image in $\End_K(\mathcal{L})$ under this representation. Thus $\{1, D_q,
D_q^2, \ldots\}$ is a free basis of the image of $H_K (q)$ in its natural
structure as a left $K[X]$-module, where $X$ acts as left multiplication with
the endomorphism $M$, and if $P\neq 0$ is written uniquely as
\begin{equation}\label{eq:P}
P = \sum_{j=0}^m p_j (M) D_q^j,\,\,
\ p_m \neq 0,
\end{equation}
for uniquely determined $p_j \in K[X]$ and $m \geq 0$, then $m$ is the order of $P$.

We will now explain the eliminant construction. We will do so in terms of the faithful representation of $H_K(q)$ as endomorphisms of $\mathcal L$ in order to stay as close as possible to the proofs in the remainder of this paper, but the reader will have no trouble formulating everything in terms of the original generators.

Let $P, Q \in H_K (q)$ be of order $m \geq 1$ and $n \geq 1$, respectively, with $P$ as in \eqref{eq:P} and
\begin{equation}\label{eq:Q}
Q = \sum_{j=0}^n q_j (M) D_q^j \,\,\, (n \geq 1, q_n \neq 0).
\end{equation}
Write, for $k = 0, \ldots, n-1$,
\[D_q^k P = \sum_{j=0}^{m+k} p_{k, j} (M) D_q^j, \textup{ with } p_{k,j} \in K[X],\]
and, for $l = 0, \ldots, m-1$, write
\[D_q^l Q = \sum_{j=0}^{n+l} q_{l, j} (M) D_q^j, \textup{ with } q_{l,j} \in K[X].\]
Using these expressions we may build up an $(m+n) \times (m+n)$-matrix with
entries in the polynomial ring $K[X,\lambda,\mu]$ in three variables over $K$, as follows. For $k = 1, \ldots, n$, the $k$-th row is given, from left to
right, by the coefficients of the increasing powers of $D_q$ in the expression
$D_q^{k-1} P - \lambda D_q^{k-1}= \sum_{j=0}^{m+k-1} p_{k-1, j} (M) D_q^j -
\lambda D_q^{k-1}$. For $k = n+1, \ldots, n+m$, the $k$-th row is given, from
left to right, by the coefficients of the increasing powers of $D_q$ in the
expression $D_q^{k-n-1} Q - \mu D_q^{k-n-1}= \sum_{j=0}^{k-1} p_{k-n-1, j} (M)
D_q^j - \mu D_q^{k-n-1}$. The determinant of this matrix is an element of
$K[X,\lambda,\mu]$ which is called the \emph{eliminant} of $P$ and $Q$. We
denote it by $\Delta_{(P, Q)} (X, \lambda, \mu)$. For clarity, we include the
following example.
\begin{ex}\label{eli32}
Let $P$ and $Q$ be as above, with $m=3$ and $n=2$.
We then have
\\
\\
$\Delta_{P,Q} (X, \lambda, \mu)=$
\[\left|\begin{array}{ccccc}
p_{0,0}(X) - \lambda & p_{0,1}(X) & p_{0,2} (X) & p_{0,3} (X) & 0\\
p_{1,0}(X)  & p_{1,1} (X)  - \lambda & p_{1,2} (X)  & p_{1,3} (X) & p_{1,4} (X)\\
q_{0,0}(X) - \mu & q_{0,1}(X) & q_{0,2} (X) & 0 & 0\\
q_{1,0}(X)  & q_{1,1} (X)  - \mu & q_{1,2} (X)  & q_{1,3} (X) & 0\\
q_{2,0}(X)  & q_{2,1} (X)  & q_{2,2}(X) - \mu & q_{2,3}(X) & q_{2,4} (X)
\end{array}\right|.\]
\end{ex}
We need a few more preparatory definitions in order to be able to state Theorem~\ref{mainthm} in its most precise form, which not only tells us that the eliminant construction yields explicit curves for commuting elements $P$ and $Q$ of $H_K(q)$, but which also gives a priori information on the maximal number of curves thus obtained, on their maximal degree and on their coefficients.

If $P$ and $Q$ are as in \eqref{eq:P} and \eqref{eq:Q}, respectively, then let
\begin{equation}\label{eq:s_def}s = n \max_{j} \deg(p_j) +  m \max_j \deg(q_j).
\end{equation}
A moment's thought shows that $s$ is an upper bound for the degree of $X$ which occurs in $\Delta_{P,Q} (X, \lambda, \mu)$, so that we can define the polynomials $\delta_i\in K[\lambda,\mu]\,\,(i=0,\ldots,s)$ by
\begin{equation}\label{eq:delta_i_def}
\Delta_{P,Q} (X, \lambda, \mu) = \sum_{i=0}^s \delta_i (\lambda, \mu) X^i.
\end{equation}
Note that the $\delta_i$ define curves over $K$ of degree at most $\max(m,n)$.
Finally, let
\begin{equation}\label{eq:t_def}
t = \frac{1}{2} n(n-1) \max_j \deg(p_j) + \frac{1}{2} m(m-1)\max_j \deg(q_j).
\end{equation}

\begin{thm}\label{mainthm}
Let $K$ be a field and $0 \neq q \in K$ be such
that $\{n\}_q =0$ if and only if $n=0$. Suppose $P$ as in \eqref{eq:P} and $Q$ as in \eqref{eq:Q} are commuting
elements of $H_K (q)$ of order $m\geq 1$ and $n\geq 1$, respectively. Let $\Delta_{P,Q} (X, \lambda, \mu)\in
K[X,\lambda,\mu]$ be the eliminant constructed as above, define $s$ as in \eqref{eq:s_def}, $\delta_i\in K[\lambda,\mu]\,\,(i=1,\ldots,s)$ as in \eqref{eq:delta_i_def}, and $t$ as in \eqref{eq:t_def}.

Then $\Delta_{P,Q} \neq 0$. In fact, $\Delta_{P,Q}$ has degree $n$ as an element of $K[X,\mu][\lambda]$ and its
non-zero coefficient of $\lambda^n$ is $(-1)^n \prod_{k=0}^{m-1} q_n(q^{k}X)$. Likewise, $\Delta_{P,Q}$ has
degree $m$ as an element of $K[X,\lambda][\mu]$ and its non-zero coefficient of
$\mu^m$ is $(-1)^m \prod_{k=0}^{n-1} p_m(q^{k} X)$. As an element of $K[\lambda,\mu][X]$, $\Delta_{P,Q}$ has degree at most $s$.
Furthermore,
\begin{enumerate}
\item if $R$ is the subring of $K$ which is generated by the coefficients of all $p_{i,j}$ and $q_{i,j}$ occurring in the matrix defining the eliminant, then the $\delta_i$ are actually elements of $R[q][\lambda,\mu]$. In fact, when viewed as polynomials in $\lambda$ and $\mu$, each coefficient of the $\delta_i$ can be written as $\sum_{l=0}^t r_l q^l$ for some $r_l\in R\,\,(l=0,\ldots,t)$;
\item at least one of the $\delta_i$ is non-zero;
\item $\delta_i (P, Q) = 0$ for all $i = 0,\ldots, s$.
\end{enumerate}
\end{thm}

\begin{rmk}
Note that parts (2) and (3) state that the eliminant construction gives at least one non-trivial curve for commuting $P$ and $Q$, and at most $s$.
Each of these curves is defined over $R[q]$, where $R$ is the ring in part (1) and where the power of $q$ -- when viewed as a formal variable -- occurring in the coefficients of these curves does never exceed $t$. Furthermore, as we had already noted, each of these curves is of degree at most $\max(m,n)$.
\end{rmk}

The reader will easily convince himself of all statements in the theorem other than (3). We will now embark on the proof of (3), which occupies the remainder of this paper.
The idea is as follows. Suppose $\lambda_0, \mu_0 \in K$ and $0 \neq
v_{\lambda_0, \mu_0} \in \mathcal{L}$ is a common eigenvector of $P$ and $Q$:
\[P v_{\lambda_0, \mu_0} = \lambda_0 v_{\lambda_0, \mu_0},\]
\[Q v_{\lambda_0, \mu_0} = \mu_0 v_{\lambda_0, \mu_0}.\]
Then the specialization $X=M, \lambda = \lambda_0, \mu = \mu_0$ of the matrix
defining the eliminant yields a matrix of commuting endomorphisms of $\mathcal{L}$ having
the vector $(v_{\lambda_0, \mu_0}, \ldots, D_q^{m+n-1} v_{\lambda_0, \mu_0})^T$
in its kernel. Since the coefficients of the matrix are from a commutative ring, multiplication from the left with the matrix of cofactors shows that $(v_{\lambda_0, \mu_0}, \ldots, D_q^{m+n-1} v_{\lambda_0, \mu_0})^T$ is annihilated by a diagonal matrix with $\Delta_{P,Q} (M, \lambda_0,
\mu_0)$ on the diagonal. In particular, $\Delta_{P,Q} (M, \lambda_0,
\mu_0) \, v_{\lambda_0, \mu_0} = 0$. Now it does not follow automatically from
this that $\Delta_{P,Q} (M, \lambda_0, \mu_0) =0$ in $H_K (q)$ since a
polynomial in $M$ might have non-trivial kernel, as the example $(M-1) \sum_n
t^n =0$ shows. However, embedding $K$ in an algebraically closed field if
necessary, we will be able to show that there exist infinitely many such pairs
$(\lambda_0, \mu_0)$ where we \emph{can} conclude that $\Delta_{P,Q} (M,
\lambda_0, \mu_0) =0$ in $H_K (q)$. For all these pairs one has
$\delta_i(\lambda_0,\mu_0)=0$ for all $i$, and the operators $\delta_i (P, Q)$
therefore have an infinite dimensional kernel. From this
Theorem~\ref{thm:cordimker} allows us to conclude that $\delta_i (P, Q) =0$ in
$H_K (q)$ for all $i$.

The first step, which consists of showing that there are infinitely many
$(\lambda_0, \mu_0)\in K\times K$ such that $\Delta(M,\lambda_0,\mu_0)=0$, is
the most involved. The idea is to exploit the fact that $v_{\lambda_0, \mu_0}$
is both in the kernel of $P - \lambda_0$ of order $m \geq 1$ and in the kernel
of the polynomial element $\Delta_{P,Q} ( M, \lambda_0, \mu_0)$ which, if it is
not zero, is not constant. This is a rare occasion. To be precise: for each $d$
we can describe the kernel of a non-constant polynomial element $p(M)$ of $H_K
(q)$ of degree at most $d$ and the action of $P- \lambda_0$ on it explicitly
enough to show that any $v_{\lambda_0, \mu_0}$ as above is in a subspace of
finite dimension which depends only on the leading coefficient of $P$ and on
$d$, but not on $\lambda_0$, $\mu_0$ or $p(M)$. This follows from
Theorem~\ref{simeigfin} below. Hence for the infinity of different pairs
$(\lambda_0, \mu_0)$ that can be shown to exist in the simultaneous point
spectrum\footnote{Here and elsewhere the point spectrum is defined as the set
of eigenvalues.}, it can, by linear independence of the corresponding
eigenvectors, only for finitely many pairs be the case that $\Delta_{P,Q} (M,
\lambda_0, \mu_0)$ is not constant. For the remaining infinite number of pairs
we must have that $\Delta_{P,Q} (M, \lambda_0,\mu_0)$ is zero.

\begin{rmk}
In the original work of Burchnall and Chaundy \cite{Burchnall&Chaundy1}, where
they consider differential operators with polynomial coefficients acting on
real or complex valued functions, the situation is considerably simpler. This
is not so much caused by the fact that they can (and do) use existence and
uniqueness results for ordinary differential equations, but by the fact that
the ordinary differentiation $D_1$ is translation invariant, whereas $D_q$
$(q\neq 1)$ is not. We will now explain why this is such a serious
complication for the strategy of the proof. The reader who is mostly interested in the established results per se can safely skip this Remark, which is primarily intended for readers who consider applying similar techniques in other cases.

It will become apparent below that, for an approach in the vein of Burchnall
and Chaundy to succeed, one needs a faithful representation of $H_K(q)$ in
which an arbitrary non-constant $P\in H_K(q)$ has an infinite point spectrum.
Without this the whole construction falls apart. How can one obtain such a
representation? In the work of Burchnall and Chaundy the smooth functions
provide such a module and the basic results about differential equations
provide the infinite point spectrum. In our general case we do not have such
results available, but there is an obvious attempt to obtain a substitute,
namely by working with formal power series. Already for Burchnall and Chaundy
themselves it would have been possible to do this and obtain the infinite point
spectrum directly, without an appeal to general theorems about differential
equations. Of course there are matters of convergence to be taken care of,
because in their proof it is necessary to evaluate solutions and their
derivatives in a point, but there is hope that a similar approach with formal
power series might somehow work in our case. However, there is an important
point here, which we have been deliberately sloppy about in the previous
sentences: a differential operator with polynomial coefficients has to be
sufficiently regular for these power series to exist as eigenfunctions, even
already as formal series. As an example, the only value of $\lambda\in\mathbb
C$ for which the operator $t^2d/dt-\lambda$ has a nontrivial kernel in the
formal power series, is $\lambda=0$. Hence in this module the point spectrum of
this operator is finite. In the case of Burchnall and Chaundy, this is not a
serious obstruction because one can simply use any point where the leading
coefficient of $P$ does not vanish as a base point for the formal power series.
In that suitably chosen module the point spectrum is infinite again and
corresponds to honest functions, as desired. The crux is that the translation
invariance of $d/dt$ is used here and that in our case this does not work any
longer. Surely the leading coefficient of $P$ can be written as a linear
combination of terms of the form $(t-a)^i$ where $a$ is chosen such that the
leading coefficient does not vanish at $a$ -- and such $a$ exist because $K$ is
infinite -- but for $q\neq 1$ the operator $D_q$ is not particulary well
behaved as far as its action on the $(t-a)^i$ is concerned. The operator $M^2
D_q$ has only zero as point spectrum in the formal power series with
coefficients in $K$ and for $q\neq 1$ there seems no way to remedy this by
choosing another base point to work with. Hence one has to pass to a larger
module, such as the formal Laurent series as we have introduced above, where
the point spectrum can be shown to be infinite again. However, in that case a
new complication appears as compared to the original context of Burchnall and
Chaundy, namely that a non-zero polynomial may have a non-trivial kernel when
acting on the Laurent series, cf.\ Proposition~\ref{baspsi}. In the sketch of
our proof preceding this remark this prohibits us from concluding that
$\Delta_{P,Q} (M, \lambda_0,\mu_0)=0$ once we know that $\Delta_{P,Q} (M,
\lambda_0,\mu_0)v_{\lambda_0,\mu_0}=0$. If $v_{\lambda_0,\mu_0}$ were a formal
power series, then this \emph{could} be concluded and the proof would be
relatively short and close to the original work of Burchnall and Chaundy, but
as explained above, as a consequence of the fact that $D_q$ is not translation
invariant for $q\neq 1$ we were forced to leave this context of formal power
series in order to ensure that the point spectrum of a non-constant element of
$H_K(q)$ is infinite.

It is in this way that it becomes necessary, in the end, to analyse the
situation in more detail and exploit the fact that $v_{\lambda_0,\mu_0}$ is not
only annihilated by $\Delta_{P,Q} (M, \lambda_0,\mu_0)$, but is also in the
kernel of an element of $H_K(q)$ of order at least one. This leads to
Theorem~\ref{simeigfin} and establishing this theorem complicates the proof
considerably as compared to the original argument by Burchnall and Chaundy.

To conclude this remark, we mention that for $q=1$, where $D_1$ is translation
invariant, it turns out that it \emph{is} possible to work with formal power
series with a suitable base point. Theorem~\ref{simeigfin} is then not needed,
but since this result is informative in the case $q=1$ as well, and since the
presentation would only be lengthened by covering this case separately, we have
chosen to give a uniform treatment with Laurent series including the easier
case $q=1$.
\end{rmk}

The structure of the remainder of the paper is as follows. In
Section~\ref{sec:dimension} we show that non-constant elements in $H_K (q)$
have finite dimensional eigenspaces and infinite spectrum when acting on $\mathcal{L}$. The finite
dimensionality will be seen to follow from the assumption that the $\{n\}_q$
are all different. Section~\ref{sec:kerpol} contains an analysis of the kernel
of non-constant polynomial elements $p(M)$ in $H_K (q)$. These kernels are
spanned by certain elements $\Psi_{\alpha, s}$ in $\mathcal{L}$ where $\alpha
\in K^*$ and $s = 1, 2, 3, \ldots$. In Section~\ref{sec:parord} we introduce a
partial ordering on the indices $(\alpha, s)$ and we analyze the action of an
arbitrary $P \in H_K (q)$ on the $\Psi_{\alpha, s}$ in terms of this partial
order. These results are then used in Section~\ref{sec:simultaneous} in order
to arrive at the finite dimensional space mentioned above.
Section~\ref{sec:proof} contains the details of the conclusion of the proof as
it has been sketched in the current section.

\section{Dimension of eigenspaces and infinity of the point spectrum}\label{sec:dimension}
For $P \in H_K (q)$, let $\sigma (P)$ denote the point spectrum of $P$ in
$\mathcal{L}$, i.e., the set of eigenvalues. Under our standing assumption that
$q \neq 0$ and $\{n\}_q \neq 0$ if $n \neq 0$, we will show that for
non-constant $P$ all eigenspaces have finite dimension
(Theorem~\ref{thm:cordimker}) and that the point spectrum is infinite
(Theorem~\ref{eigspec}). Although it is not needed for the proof of
Theorem~\ref{mainthm}, as a side result we will also establish in
\eqref{eq:uniformbound} a uniform upper bound for the dimension of all
eigenspaces of a fixed non-constant $P$.

Let $P = \sum_{j=0}^m p_j (M)
D_q^j\,\,(m \geq 0)$,
where $p_j (M) = \sum_i p_{j,i} M^i$ and $p_m\neq 0$. Then
clearly, for all $k \in \mathbb{Z}$,
\[P t^k =
\sum_{d} \left(\sum_{i-j = d} p_{j,i} \{k\}_{q} \{k-1\}_{q} \ldots \{k-j+1\}_{q}\right)
t^{d+k}.\] Here the product $\{k\}_{q} \{k-1\}_{q} \ldots \{k-j+1\}_{q}$ should
be interpreted as $1$ if $j =0$.

Let $\beta_d (k) = \sum_{i-j = d} p_{j,i} \{k\}_{q} \{k-1\}_{q} \ldots
\{k-j+1\}_{q} \,\,(k, d \in \mathbb{Z})$. The function $\beta_d : \mathbb{Z}
\rightarrow K$ describes the action of the homogeneous part of $P$ of degree
$d$ on $t^k$. Say that a homogeneous degree $d$ \emph{occurs in $P$} if there
exist $i, j$ with $i-j = d$ such that $p_{j,i} \neq 0$. Say that a homogeneous
degree $d$ \emph{occurs in $P$ with a differentiation} if there exist $i, j$
with $i-j =d$, $p_{j,i} \neq 0$ and $j \geq 1$. Obviously, only finitely many
homogeneous degrees occur in $P$. Now iteration of the recursion $\{n-1\}_{q} =
\frac{\{n\}_q -1}{q}$ shows that there exist polynomials $r_j$ of precise
degree $j$, with coefficients in $\mathbb{Z} [q, q^{-1}]$, such that $\{k\}_{q}
\{k-1\}_{q} \ldots \{k-j+1\}_{q} = r_j (\{k\}_{q})$ for all $k$. Hence $\beta_d
(k) = \sum_{i-j = d} p_{j, i}r_j (\{k\}_{q})$ is a polynomial in $\{k\}_{q}$ of
maximal degree $m$. Suppose that the homogeneous degree $d$ occurs in $P$. If
it does not occur with a differentiation, then $\beta_d (k) = p_{0,d}$ is a
non-zero constant function of $k$. If it does occur with a differentiation
then, since the degree of $r_j$ is precisely $j$, $\beta_d (k)$ is a
non-constant polynomial in $\{k\}_{q}$ of maximal degree $m$. Since the
$\{n\}_q$ are all different, $\beta_d (k)$ therefore assumes each value in $K$
at most $m$ times. In particular, it has a finite number of zeroes in
$\mathbb{Z}$. This establishes the following result.
\begin{lem}\label{betalem}
Suppose $P\neq 0$. Then the following are equivalent for $d \in \mathbb{Z}$:
\begin{enumerate}
\item the homogeneous degree $d$ occurs in $P$; \item $\beta_d \neq 0$; \item
$\beta_d$ has only finitely many zeroes in $\mathbb{Z}$.
\end{enumerate}
If the homogeneous degree $d$ occurs in $P$, then
there are two possibilities:
\begin{enumerate}
\item $\beta_d$ is a non-zero constant.
This happens precisely when $d$ occurs in $P$,
but not with a differentiation.
\item $\beta_d$ is not constant. This happens
precisely when $d$ occurs with a differentiation.
In this case, $\beta_d$ has at most $m$ zeroes
in $\mathbb{Z}$ and its range is countably infinite.
\end{enumerate}
\end{lem}

We will now analyze the kernel of $P$.
The coefficient of $t^j$ in $P \sum_n a_n t^n$
is clearly equal to $\sum_n a_n \beta_{j-n} (n)$,
so the series $\sum_n a_n t^n$ is in the kernel
of $P$ if and only if $\sum_n \beta_{j-n}(n) a_n =0$
for all $j$. The structure of this system becomes
more transparent if we write
$\gamma_{k,l} = \beta_{k-l} (l)$
$(k, l \in \mathbb{Z})$; it then reads as $\sum_l \gamma_{k,l} a_l =0$ for
all $k \in \mathbb{Z}$. Let $\Gamma$ be the
matrix $(\gamma_{k,l})$, where we think of $\Gamma$
as being realized on $\mathbb{Z}^2$, placing the
entry $\gamma_{k,l}$ in the lattice point $(l,k)$:

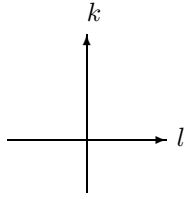
\begin{picture}(80,80)(-120,0)
\put(40,65){$k$}
\put(74,18){$l$}
\put(10,20){\vector(1,0){60}}
\put(40,0){\vector(0,1){60}}
\end{picture}

The equation $\sum_l \gamma_{k,l} a_l =0$ then corresponds to the rows in
$\Gamma$ at horizontal level $k$ acting on an infinite vector $(\ldots, a_{-2},
a_{-1}, a_0, a_1, a_2, \ldots)$ in the usual way. We now look at the matrix
$\Gamma$ along a diagonal $k = l+d$ with $d$ fixed. For such pairs $(k, l)$,
one has $\gamma_{k,l} = \beta_d (l)$. Since only finitely many $d$ occur in
$P$, $\Gamma$ is a band matrix and, moreover, according to Lemma~\ref{betalem}
each diagonal is either identically zero or else contains at most $m$ zeroes.

Suppose $P \neq 0$ and define $d_{\max} = \max\{d: \beta_d \neq 0\}$ and
$d_{\min} = \min \{d: \beta_d \neq 0\}$; these integers correspond to the upper
and lower boundary diagonal of the band in $\Gamma$, respectively. If $d_{\max}
= d_{\min}$, so that there is only one diagonal to consider, then $\dim \ker P
= \# \{k : \beta_d (k) =0\} \leq m$. If $d_{\max} > d_{\min}$, then, since each
of the boundary diagonals contain only finitely many zeroes, it is possible to
determine a (not uniquely determined) finite submatrix $\widetilde{\Gamma}$ as
indicated:

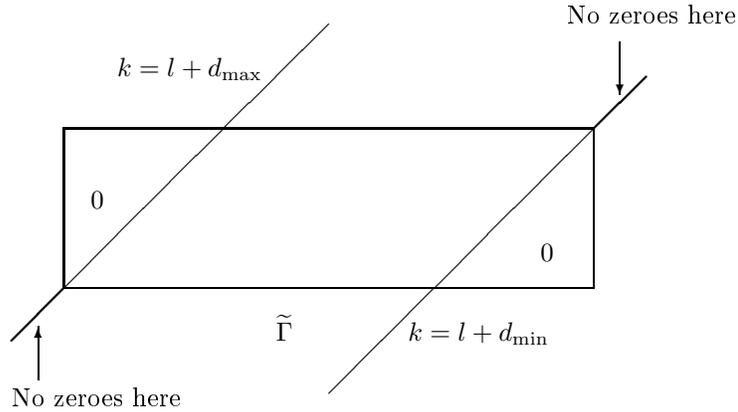
\begin{picture}(100,180)(0,-110)
\put(60,-50){\framebox(200,60)}
\put(60,-50){\line(1,1){100}}
\thicklines
\put(40,-70){\line(1,1){20}}
\thinlines
\put(160, -90){\line(1,1){100}}
\thicklines
\put(260,10){\line(1,1){20}}
\put(140, -70){$\widetilde{\Gamma}$}
\put(190, -70){$k = l + d_{\min}$}
\put(80, 30){$k = l + d_{\max}$}
\thinlines
\put(50, -85){\vector(0,1){20}}
\put(40, -95){$\textup{No zeroes here}$}
\put(270, 43){\vector(0,-1){20}}
\put(250, 50){$\textup{No zeroes here}$}
\put(240, -40){$0$}
\put(70, -20){$0$}
\end{picture}

The relevant features here are that the only non-zero elements occur on the
boundary diagonals and in the band between them, and that there are no zeroes on the
indicated lower part of the upper diagonal and on the indicated upper part of
the lower diagonal. A moment's thought shows that the kernel of $P$ and the
kernel of $\widetilde{\Gamma}$ are isomorphic: an isomorphism is given by
selecting the coordinates corresponding to all columns of $\widetilde{\Gamma}$
from the infinite vector representing an element of the kernel of $P$ and thus obtain an element of the kernel of $\widetilde\Gamma$. The injectivity and the surjectivity of this map are both consequences of the
non-zero elements on the boundary diagonals as indicated. Namely, on the lower
diagonal they enable the necessary unique downward extension of an element in the
kernel of $\widetilde{\Gamma}$ to an infinite column vector in the kernel of $\Gamma$, and on
the upper diagonal they enable the necessary unique upward extension.

We conclude that, if $d_{\max} > d_{\min}$, then $P$ has a finite dimensional and non-trivial kernel. Since we had already concluded that the kernel is finite dimensional in the case $d_{\max}=d_{\min}$, we have arrived at the following result.

\begin{thm}\label{thm:cordimker}
If $P \in H_K (q)$, then $\dim \ker (P) = \infty$ if and only if $P=0$.
\end{thm}

In fact, although we will not need this, we can be more precise. Let $N_{\max}
= \# \{l : \beta_{\max} (l) =0\}$, $N_{\min} = \# \{l : \beta_{\min} (l) =0\}$.
By Lemma~\ref{betalem}, $N_{\max}, N_{\min} \leq m$. We may assume that all the
zeroes on the boundary diagonals lie in $\widetilde{\Gamma}$, and since the
number of non-zero elements in each of the parts of the boundary diagonals that
are contained in $\widetilde{\Gamma}$ gives a lower bound for the rank of
$\widetilde{\Gamma}$, one easily derives that
\[\dim \ker P \leq d_{\max} - d_{\min} + \min (N_{\max}, N_{\min}).\]
Note that this is also true if $d_{\max} = d_{\min}$ (with equality). It is
even more elementary to see that $d_{\max} - d_{\min} \leq \dim \ker P$, and we
thus obtain the following result under our standard assumption on the
$\{n\}_q$.

\begin{pro}\label{prop:dimker}
If $0 \neq P \in H_K (q)$, then \[d_{\max} - d_{\min} \leq \dim \ker P \leq d_{\max} - d_{\min} + \min(N_{\max}, N_{\min}) \leq d_{\max} - d_{\min} +m.\]
\end{pro}

It is now also easy to see that, for non-constant $P$, there is a uniform bound
for $\dim \ker (P - \lambda)$. Since this corresponds to adding $- \lambda$ to
the diagonal of $\Gamma$, the relevant numbers $d_{\max} (\lambda)$ and
$d_{\min} (\lambda)$ can attain only a finite number of values, the number of
which depends on the position of the boundary diagonals in $\Gamma$ and, also,
if the main diagonal $k = l$ is one of the boundary diagonals of the band in
$\Gamma$, on $\Gamma$ being constant along this main diagonal or not.
Distinguishing various possibilities one obtains that, for $P \in H_K (q)$ not
constant,
\begin{equation}\label{eq:uniformbound}
\dim \ker (P - \lambda) \leq |d_{\max}| + |d_{\min}| + m
\end{equation}
for all $\lambda \in K$.

Returning to the main line, we will now establish an important result.

\begin{thm}\label{eigspec}
If $P \in H_K (q)$ is not constant, then $\sigma (P)$ is infinite.
\end{thm}
\begin{proof}
If a homogeneous degree $d \neq 0$ occurs in $P$, then the matrix $\Gamma$ has
a non-vanishing diagonal which is not the main diagonal. Therefore, the matrix
for $P - \lambda$ has two non-vanishing diagonals for all $\lambda\in K$ except
at most one value. Since $d_{\max} - d_{\min}>0$ for all matrices corresponding
to such non-exceptional $\lambda$, and we had already observed in Remark
\ref{nonunroot} that $K$ must be infinite, the theorem is established in this
case. If $P$ is homogeneous of degree zero then $\sigma (P) = \{\beta_0 (k) : k
\in \mathbb{Z}\}$ which is (countably) infinite according to
Lemma~\ref{betalem}, since $P$ is not constant.
\end{proof}

\section{The kernel of polynomial elements of $H_K (q)$}\label{sec:kerpol}
Throughout this section we assume that $K$ is algebraically closed, in addition
to our standard assumption that $q \neq 0$ and $\{n\}_q \neq 0$ if $n \neq 0$.
For arbitrary non-zero $p \in K[X]$, we will describe the kernel of the corresponding endomorphism $p(M)$ of $\mathcal L$ in terms of infinite Jordan blocks corresponding to the eigenvalues of the endomorphism $M$ of $\mathcal L$.

If $p(M) = c M^i$ for some $i \geq 0$ and $c \neq 0$, then $p(M)$ is an automorphism of $\mathcal{L}$. Thus, if $p(X) = c X^{e_0} \prod_{i \neq 0} (X- \alpha_i)^{e_i}$ is the factorization of $p$ with multiplicities, where $\alpha_i \in K^*$, then $\Ker (p(M))$ = $\Ker (\prod_{i \neq 0} (M-\alpha_i)^{e_i})$.
Now it is an easy exercise to show that for $\alpha \in K^*$ the map $(M-\alpha) : \mathcal{L} \rightarrow \mathcal{L}$ is surjective and that the element
\begin{equation}
\Psi_{\alpha, 1} = \sum_n \left(\frac{t}{\alpha}\right)^n \label{basker}
\end{equation}
is a basis for $\Ker (M- \alpha)$. We conclude that $\dim \Ker (p(M)) = \sum_{i \neq 0} e_i$.

We choose inductively $\Psi_{\alpha, s} (s = 2, 3, \ldots)$ in $\mathcal{L}$
such that $(M-\alpha) \Psi_{\alpha, s} = \Psi_{\alpha, s-1}$. The elements
$\Psi_{\alpha, s}$ corresponding to this infinite Jordan block are by no means
unique, but this is not serious and we fix such a choice once and for all, for
all $s \geq 2$ and $\alpha \in K^*$. The only normalization which we impose
is~\eqref{basker}.

Note that
\begin{equation}\left\{\begin{array}{lr}
(M-\alpha)^s \Psi_{\alpha,s} = 0, \,\,(s = 1, 2, \ldots ; \alpha \in K^*),\\
(M-\alpha)^{s-1} \Psi_{\alpha, s} = \Psi_{\alpha, 1},
\end{array}\right. \label{psigrejs}\end{equation}
and that $M^i \Psi_{\alpha, s} = \alpha^i \Psi_{\alpha,s} + \sum_{r < s}
c_{r,s} \Psi_{\alpha,r}$. We will repeatedly encounter similar formulas
containing a summation where the only role of the summation is to indicate the
subspace containing the sum. As a shorthand notation we will allow ourselves to
suppress the dependence of the scalars on the indices and write this as
\[\sum_{r < s} c \Psi_{\alpha, r},\]
and similarly in other situations. With this convention we have
\begin{equation}
p(M) \Psi_{\alpha,s} = p(\alpha) \Psi_{\alpha,s} + \sum_{r<s} c \Psi_{\alpha,
r} \,\, (\alpha \in K^*, s = 1, 2, \ldots\,, p \in K[X]).\label{konv}
\end{equation}
\begin{pro}\label{baspsi}
If $0 \neq p \in K[X]$ factors with multiplicities as
\[p(X) = c X^{e_0}\prod_{i \neq 0} (X - \alpha_i)^{e_i},\]
then:
\begin{enumerate}
\item $p(M) : \mathcal{L} \rightarrow \mathcal{L}$ is surjective;
\item $\dim \Ker(p(M)) = \sum_{i \neq 0} e_i$;
\item $\bigcup_{i \neq 0} \{\Psi_{\alpha_i,1},
\ldots, \Psi_{\alpha_i,e_i}\}$ is a basis for $\Ker (p(M))$.
\end{enumerate}
\end{pro}
\begin{proof}
The first statement is clear, and the formula for $\dim \Ker(p(M))$ was already
noted above. From the first part of (2) we see that the $\Psi_{\alpha_i, k}$
are in the kernel of $p(M)$ for $1 \leq k \leq e_i$. The linear independence
follows by the standard argument: if $\sum_{i, k} \lambda_{i,k} \Psi_{\alpha_i,
k} =0$, suppose that not all coefficients are zero. Then choose indices $i_0$
and $k_0$ such that $\lambda_{i_0, k_0} \neq 0$, but $\lambda_{i_0, k} = 0$ for
all $k < k_0$. Applying
\[[\prod_{\substack {i \neq 0 \\ i \neq i_0}} (M - \alpha_i)^{e_i}] (M-\alpha_{i_0})^{k_0 -1}\]
one sees, when looking at the second part of \eqref{psigrejs} and \eqref{konv},
that
\[\lambda_{i_0, k_0} \prod_{\substack{i \neq 0 \\ i \neq i_0}} (\alpha_{i_0} - \alpha_i)^{e_i} \Psi_{\alpha_{i_0}, 1} =0.\]
Hence $\lambda_{i_0, k_0} =0$ after all and we have a contradiction.
\end{proof}
The argument in the above proof shows in fact the following.
\begin{pro}\label{almlinin}
If $\alpha_1, \ldots, \alpha_s$ are $s$ different elements in $K^*$, and $e_1, \ldots, e_s \geq 1$, then the elements $\Psi_{\alpha_1, 1}, \ldots,
\Psi_{\alpha_1, e_s}, \ldots, \Psi_{\alpha_s, 1}, \ldots, \Psi_{\alpha_s, e_s}$
are linearly independent over $K$.
\end{pro}

\section{Partial order}\label{sec:parord}
Throughout this section, which is a preparation for
Section~\ref{sec:simultaneous}, we assume that $K$ is algebraically closed. We
will analyze the action of an arbitrary $P \in H_K (q)$ on the
$\Psi_{\alpha,s}$ from the previous section and see that the results are related to a partial order on the
indices $(\alpha, s)$ which we now introduce. Let us take $\mathbb{N} = \{1,
2, \ldots\}$ as convention.
\begin{defn}\label{partord}
On $K^* \times \mathbb{N}$, define
\begin{enumerate}
\item if $q = 1$: $(\alpha, r) \leq (\beta, s)$ if and only if $\alpha = \beta$ and $r \leq s$.
\item if $q \neq 1$: $(\alpha, r) \leq (\beta, s)$ if and only if $\beta = \frac{\alpha}{q^j}$ for some $j > 0$, or if $\alpha = \beta$ and $r \leq s$.
\end{enumerate}
\end{defn}
It is easily checked that this is a partial ordering under our assumption on the $\{n\}_q$. Note that being comparable for this partial order is an equivalence relation on $K^* \times \mathbb{N}$.
There is a natural $\mathbb{Z}$-action on $K^*$ by multiplication with powers of $q$ (which is a trivial action if $q = 1$), and two elements $(\alpha, r)$ and $(\beta, s)$ are comparable precisely when $\alpha$ and $\beta$ are in the same $\mathbb{Z}$-orbit.
An equivalence class of mutually comparable pairs is of the form \[\bigcup_{\substack{j \in \mathbb{Z} \\ s \in \mathbb{N}}} (q^j \alpha, s)\] for some $\alpha \in K^*$, which is uniquely determined only if $q = 1$. The verification of the following lemma is routine.
\begin{lem}\label{parordfacts}
For the partial ordering on $K^* \times \mathbb{N}$ defined above the following hold.
\begin{enumerate}
\item Suppose $q =1$. Then for all $m \geq 0$ and all $(\alpha, r), (\beta, s) \in K^* \times \mathbb{N}$, one has $(\alpha, r+m) \leq (\beta, s+m)$ if and only if $(\alpha, r) \leq (\beta, s)$.
\item Suppose $q \neq 1$. Then for all $m \in \mathbb{Z}$ and all $(\alpha, r), (\beta, s) \in K^* \times \mathbb{N}$, one has $(\frac{\alpha}{q^m}, r) \leq (\frac{\beta}{q^m}, s)$ if and only if $(\alpha, r) \leq (\beta, s)$.
\end{enumerate}
\end{lem}
Note that \eqref{konv} can be written as
\begin{equation}p(M) \Psi_{\alpha, s} = p(\alpha) \Psi_{\alpha, s} + \sum_{(\beta,r) \in [(\alpha, 1), (\alpha, s))} c \Psi_{\beta,r}, \label{intval}
\end{equation}
where $[(\alpha, 1), (\alpha,s))$ is a left-closed and right-open order
interval. The common thread in this section is to establish that various
elements of $\mathcal L$ can similarly be regarded as a sum of a leading term
corresponding to a certain index $(\alpha,s)$, and a remainder which is a sum of terms
corresponds to various indices lying strictly below $(\alpha,s)$ in the partial
order on the index set. This observation will be crucial in the proof of
Theorem~\ref{simeigfin}. For reasons of notational simplicity we will not formulate the results as in \eqref{intval}, and it would in fact have been possible to introduce the partial order only in Section~\ref{sec:simultaneous} on the occasion of the proof of Theorem~\ref{simeigfin}, but the reader may find it helpful to view the results in the present section in the light of this partial ordering already.

We turn to the action of $D_q$ on the $\Psi_{\alpha,s}$. We start with the case $q \neq 1$, which is the most complicated. A routine computation gives
\begin{equation}
D_q \Psi_{\alpha, 1} = \frac{q}{\alpha (q-1)} \Psi_{\frac{\alpha}{q}, 1} - \frac{1}{\alpha (q-1)} \Psi_{\alpha, 1}.\label{dqpaker}
\end{equation}
\begin{pro}\label{indqneq1}
If $q \neq 1$, then
\[D_q \Psi_{\alpha, s} = \frac{q^{2-s}}{\alpha (q-1)} \Psi_{\frac{\alpha}{q},s} + \sum_{r <s} c \Psi_{\frac{\alpha}{q}, r} + \sum_{r \leq s} c \Psi_{\alpha,r}.\]
\end{pro}
\begin{proof}
By induction. The case $s =1$ follows from \eqref{dqpaker} and for the
induction step we argue as follows, using the relation $q M D_q = D_q M -I$ in
the first equality and assuming that $s\geq 2$:
\begin{align*}
&(M- \alpha)^s (M - \frac{\alpha}{q})^s D_q \Psi_{\alpha, s} = \frac{1}{q} (M - \alpha)^s (M - \frac{\alpha}{q})^{s-1} (D_q M - I - \alpha D_q) \Psi_{\alpha,s}\\
&\stackrel{s \geq 2}{=} \frac{1}{q} (M - \alpha)^s (M - \frac{\alpha}{q})^{s-1} (D_q (\alpha \Psi_{\alpha, s} + \Psi_{\alpha, s-1}) - \Psi_{\alpha, s}- \alpha D_q \Psi_{\alpha, s})\\
&=\frac{1}{q} (M - \alpha)^s (M - \frac{\alpha}{q})^{s-1} (D_q \Psi_{\alpha, s-1} - \Psi_{\alpha, s}) \stackrel{\eqref{psigrejs}}{=} \frac{1}{q} (M - \alpha)^s (M - \frac{\alpha}{q})^{s-1} D_q \Psi_{\alpha, s-1}\\
&\stackrel{\mathrm{ind.}}{=} \frac{1}{q} (M-\alpha)^s (M - \frac{\alpha}{q})^{s-1} \left[\frac{q^{3-s}}{\alpha (q-1)} \Psi_{\frac{\alpha}{q}, s-1} + \sum_{r < s-1} c \Psi_{\frac{\alpha}{q},r} + \sum_{r \leq s-1} c \Psi_{\alpha, r}\right]\\
&\stackrel{\eqref{psigrejs}}{=} 0.
\end{align*}
From Proposition~\ref{baspsi} we conclude that
\[D_q \Psi_{\alpha, s} = c_0 \Psi_{\frac{\alpha}{q},s} + \sum_{r < s} c \Psi_{\frac{\alpha}{q}, r} + \sum_{r \leq s} c \Psi_{\alpha, r}.\]
If we apply $(M- \frac{\alpha}{q})^{s-1}$ to this equation, then using
\eqref{psigrejs} and \eqref{konv} we see that the right hand side gives \[c_0
\Psi_{\frac{\alpha}{q}, 1} + \sum_{r \leq s} c \Psi_{\alpha, r}.\] The left
hand side gives
\begin{align*}
&(M-\frac{\alpha}{q})^{s-1} D_q \Psi_{\alpha, s} = \frac{1}{q} (M - \frac{\alpha}{q})^{s-2} (D_q \Psi_{\alpha, s-1} - \Psi_{\alpha, s})\\
&\stackrel{\eqref{konv}}{=} \frac{1}{q}(M - \frac{\alpha}{q})^{s-2} D_q \Psi_{\alpha, s-1} + \sum_{r \leq s} c \Psi_{\alpha, r}\\
&\stackrel{\mathrm{ind.}}{=} \frac{1}{q} (M - \frac{\alpha}{q})^{s-2} \left[\frac{q^{3-s}}{\alpha (q-1)} \Psi_{\frac{\alpha}{q}, s-1} + \sum_{r < s-1} c \Psi_{\frac{\alpha}{q},r} + \sum_{r \leq s-1} c \Psi_{\alpha, r}\right] + \sum_{r \leq s} c \Psi_{\alpha, r}\\
&= \frac{q^{2-s}}{\alpha (q-1)} \Psi_{\frac{\alpha}{q}, 1} + \sum_{r \leq s} c \Psi_{\alpha, r}.
\end{align*}
By Proposition~\ref{almlinin}, comparing completes the induction step.
\end{proof}
Iterating this result, we conclude that, for $\alpha \in K^*, j \geq 0$, and
$s\geq 1$,
\[D_q^j \Psi_{\alpha, s} = \frac{q^{\frac{j ( j - 2s +3)}{2}}}{\alpha^j (q-1)^j} \Psi_{\frac{\alpha}{q^j}, s} + \sum_{r < s} c \Psi_{\frac{\alpha}{q^j},r} + \sum_{\substack{i<j \\ r \leq s}} c \Psi_{\frac{\alpha}{q^i},r},\]
and then
\begin{align*}
&p_j (M) D_q ^j \Psi_{\alpha, s} = \frac{q^{\frac{j ( j - 2s +3)}{2}}}{\alpha^j (q-1)^j} p_j(\frac{\alpha}{q^j}) \Psi_{\frac{\alpha}{q^j}, s} +  \sum_{r < s} c \Psi_{\frac{\alpha}{q^j},r} + \sum_{\substack{i<j \\ r \leq s}} c \Psi_{\frac{\alpha}{q^i},r},\\
&(\alpha \in K^*, j \geq 0, s \geq 1, p_j \in K[X]).
\end{align*}
The following result, which is the basic ingredient in the proof of
Theorem~\ref{simeigfin} if $q\neq 1$, is now clear.
\begin{pro}\label{actpsi}
If $q \neq 1$, and $P = \sum_{j=0}^m p_j (M) D_q^j \,\,(m \geq 0)$ with $p_m \neq 0$, then for all $\alpha \in K^*$ and $s \geq 1$, we have
\[P \Psi_{\alpha,s} = \frac{q^{\frac{m ( m - 2s +3)}{2}}}{\alpha^m (q-1)^m} p_m(\frac{\alpha}{q^m}) \Psi_{\frac{\alpha}{q^m}, s} +  \sum_{r < s} c \Psi_{\frac{\alpha}{q^m},r} + \sum_{\substack{i<m \\ r \leq s}} c \Psi_{\frac{\alpha}{q^i},r}.\]
\end{pro}
We now take care of the case $q=1$, where it is easier to derive the analogue of Proposition~\ref{actpsi}.
\begin{pro}\label{actpsiq1}
For $\alpha \in K^*$ and $s=1, 2, \ldots$ we have
\[D_1 \Psi_{\alpha,s} = -s \Psi_{\alpha, s+1} + \sum_{r < s+1} c \Psi_{\alpha, r}.\]
\end{pro}
\begin{proof}
For $s = 1$ we use $M D_1 = D_1 M - I$ to see that
\begin{align*}
(M-\alpha)^2 D_1 \Psi_{\alpha,1} &= (M - \alpha) (D_1 M - I - \alpha D_1)
\Psi_{\alpha,1}\\ &= (M-\alpha) (\alpha D_1 \Psi_{\alpha,1} - \Psi_{\alpha,1} -
\alpha D_1 \Psi_{\alpha,1})\\ &= 0.
\end{align*}
Hence, by Proposition~\ref{baspsi}, $D_1 \Psi_{\alpha, 1} = c_0 \Psi_{\alpha,
2} + c \Psi_{\alpha, 1}$ and applying $M-\alpha$ shows that $-\Psi_{\alpha, 1}
= c_0 \Psi_{\alpha,1}$, establishing the case $s=1$. The induction step is
similar to the one in the proof of Proposition~\ref{indqneq1}.
\end{proof}
Thus we have the following analogue of Proposition~\ref{actpsi}, which is the
basic ingredient in the proof of Theorem~\ref{simeigfin} if $q=1$.
\begin{pro}\label{actpsiq2}
If $q = 1$, and $P = \sum_{j=0}^m p_j (M) D_q^j \,\, (m \geq 0)$ with $p_m \neq 0$, then for all $\alpha \in K^*$ and $s \geq 1$ we have
\[P \Psi_{\alpha,s} = p_m (\alpha) (-1)^m s (s+1) \ldots (s+m-1) \Psi_{\alpha,s+m} + \sum_{r < s+m} c \Psi_{\alpha, r}.\]
\end{pro}
\section{Simultaneous eigenspaces}\label{sec:simultaneous}
After the preparations in the sections~\ref{sec:kerpol} and~\ref{sec:parord} we
can now establish the following result, which is vital for the proof of
Theorem~\ref{mainthm}. We assume $K$ is algebraically closed.
\begin{thm}\label{simeigfin}
Let $P = \sum_{j=0}^m p_j (M) D_q^j$ with $m \geq
1$ and $p_m \neq 0$. Suppose $d \geq 1$. Then
there exists a finite dimensional subspace
$\mathcal{L}_{P,d}$ of $\mathcal{L}$, which depends
on $p_m$ and $d$ only, such that $v \in
\mathcal{L}_{P,d}$ whenever $(P-\lambda)v = 0$ and
$(p(M) - \mu)v = 0$ for some $\lambda, \mu \in K$
and some non-constant $p \in K[X]$ of degree at most $d$.
\end{thm}
\begin{proof}
From Propositions~\ref{almlinin}, \ref{actpsi} and \ref{actpsiq2} we know that
the sum
\[\bigoplus_{\substack{\alpha \in K^* \\ s \geq 1}} K \Psi_{\alpha,s}\]
is indeed direct and that it is an $H_K (q)$-submodule of $\mathcal{L}$. By
Proposition~\ref{baspsi} it contains the kernels of all non-zero polynomial
elements in $H_K(q)$. Hence, if $v \in \mathcal{L}$ is as in the theorem,
\[v = \sum_{\substack{\alpha \in K^*
\\ s \geq 1}} \xi_{\alpha,s} \Psi_{\alpha,s}\] for some scalars $\xi_{\alpha,s}$.
We will establish that the only possible pairs $(\alpha,s)$ with
$\xi_{\alpha,s} \neq 0$ lie in some finite set which depends on $d$ and $p_m$
only, and this clearly implies the theorem.

To facilitate terminology, say that $(\alpha,s)$ occurs in $v$ if
$\xi_{\alpha,s} \neq 0$, and that $\alpha$ (resp.\ $s$) occurs in $v$ if there
exists $s$ (resp.\ $\alpha$) such that $(\alpha,s)$ occurs in $v$. Let
\[\mathcal{O}_v = \{(\alpha,s) : (\alpha,s) \textup{ occurs in } v\}.\]
We know that $\mathcal{O}_v$ is a finite set and may assume that it
is not empty. Clearly, if $s$ occurs in $v$, then $s \leq d$ by Propositions~\ref{baspsi}
and~\ref{almlinin}, so it remains to restrict the possibly occurring values of
$\alpha$.

First we consider the case $q = 1$. Choose an element $(\alpha_0, s_0) \in
\mathcal{O}_v$ which is a maximal element of $\mathcal{O}_v$ in the partial
order. From Proposition~\ref{actpsiq2} we have, using that $m\geq 1$,
\[(P-\lambda) (\xi_{\alpha_0, s_0} \Psi_{\alpha_0, s_0}) = c_0 \xi_{\alpha_0, s_0} p_m (\alpha_0) \Psi_{\alpha_0, s_0+m} + \sum_{r < s_0 +m} c \Psi_{\alpha_0, r}\]
for some non-zero $c_0$ (recall that $\fchar K =0$ if $q = 1$).
We claim that none of the other elements $(\alpha, s)$ of $\mathcal{O}_v$ contributes to the coefficient of $\Psi_{\alpha_0, s_0 +m}$ in $(P-\lambda) v$.
Indeed, since the indices of the terms of $(P - \lambda) \Psi_{\alpha, s}$ in Proposition~\ref{actpsiq2} all lie in the order interval
\[[(\alpha,1), (\alpha, s+m)]\]
we would then have that $(\alpha_0, s_0 +m) \leq (\alpha, s+m)$, hence
$(\alpha_0, s_0) \leq (\alpha, s)$ by Lemma~\ref{parordfacts}. But then
$(\alpha, s) = (\alpha_0, s_0)$ by maximality. We conclude that $p_m (\alpha_0)
=0$ for all such maximal elements of $\mathcal{O}_v$. Since each element
$(\alpha, s)$ of $\mathcal{O}_v$ is dominated by a maximal element $(\alpha,
s+j)$ of $\mathcal{O}_v$ for some $j \geq 0$, we conclude that the only
$\alpha$ that can occur in $v$ are roots of $p_m$. This establishes the theorem
for $q=1$. Although we will not use this in the sequel, the argument actually shows that for $q=1$ one can take \[\mathcal{L}_{P,d} = \bigoplus_{\substack{\alpha \in K^* : p_m (\alpha) = 0 \\ s =1, \ldots, d}} K
\Psi_{\alpha,s}.\]

The case $q \neq 1$ is more involved. We start by establishing a fact which we will use a number of times.
Suppose that $(\alpha_0, s_0) \in \mathcal{O}_v$ but that none of the other indices in the
order interval $[(\alpha_0, s_0), (\frac{\alpha_0}{q^{m+1}},1))$ is in $\mathcal{O}_v$ (we will refer to this as property NOI). Then we must have $p_m(\frac{\alpha_0}{q^m})=0$.

In order to see this, note that from Proposition~\ref{actpsi} we have, since
$m\geq 1$, \[(P-\lambda)(\xi_{\alpha_0, s_0} \Psi_{\alpha_0, s_0}) = c_0
\xi_{\alpha_0, s_0} p_m (\frac{\alpha_0}{q^m}) \Psi_{\frac{\alpha_0}{q^m}, s_0}
+ \sum_{r < s_0} c \Psi_{\frac{\alpha_0}{q^m}, r} + \sum_{\substack{i<m \\ r
\leq s_0}} c \Psi_{\frac{\alpha_0}{q^i}, r}\] for some non-zero $c_0$. We claim
that none of the other pairs $(\alpha, s) \in \mathcal{O}_v$ contributes to the
coefficient of $\Psi_{\frac{\alpha_0}{q^m}, s_0}$ in $(P-\lambda)v$. Indeed, if
some $(\alpha, s) \in \mathcal{O}_v$ contributes, then, since all indices of
the terms of $(P- \lambda) \Psi_{\alpha, s}$ in Proposition~\ref{actpsi} lie in
the order interval $[(\alpha,1), (\frac{\alpha}{q^m}, s)]$, we would have
$(\alpha, 1) \leq (\frac{\alpha_0}{q^m}, s_0)$ and $(\frac{\alpha_0}{q^m},s_0)
\leq (\frac{\alpha}{q^m}, s)$. By Lemma~\ref{parordfacts} the second inequality
implies that $(\alpha_0, s_0) \leq (\alpha,s)$. From the first inequality we
know that $\alpha = \frac{\alpha_0}{q^m} q^j$ for some $j \geq 0$, hence
$\alpha = \frac{\alpha_0}{q^{m+1}} q^{j+1}$ with $j+1 > 0$, so that $(\alpha,s)
< (\frac{\alpha_0}{q^{m+1}}, 1)$. Hence $(\alpha,s) \in [(\alpha_0, s_0),
(\frac{\alpha_0}{q^{m+1}}, 1))$ and $(\alpha,s) = (\alpha_0, s_0)$ by
assumption. Since there are no other contributions we must have
$p_m(\frac{\alpha_0}{q^m})=0$ as asserted.

Before we proceed, let us introduce some notation as a preparation. If $\beta$ and $\widetilde\beta$ are non-zero roots of $p_m$, let us say that $\beta\leq\widetilde\beta$ if $\beta = q^j \widetilde\beta$ for some $j
\geq 0$. This introduces a partial ordering on the set of all non-zero roots of $p_m$, and we let
$\beta_1, \ldots, \beta_h$ denote the maximal elements (where $h \leq \deg
p_m$) in this set. Hence each non-zero root $\beta$ of $p_m$ can be written
as $q^{j(\beta)} \beta_{i(\beta)}$ for some uniquely determined $j(\beta) \geq
0$ and $i(\beta)\in \{1, \ldots, h\}$. We let $J = \max \{j(\beta) : \beta\neq 0 \textup{ and }p_m
(\beta) = 0\}$ denote the maximal degree which is needed.

Continuing with the proof we note that, if $(\alpha_0, s_0)$ is a maximal element of $\mathcal{O}_v$, then
certainly $(\alpha_0, s_0)$ has property NOI, hence
$\frac{\alpha_0}{q^m}$ is of the form $q^j\beta_i$ for some $1\leq j\leq J$ and
$i \in \{1, \ldots, h\}$. Since each element $(\alpha, s)$ in $\mathcal{O}_v$
is dominated by a maximal element of the form $(\frac{\alpha}{q^k},r)$ for some
$k \geq 0$ and $r \geq 1$, we conclude that each $\alpha$ that occurs in $v$
must be of the form $\alpha = q^l \beta_i$ for some $l \geq 0$ and $i \in \{1,
\ldots, h\}$. Thus the list of possibly occurring values of $\alpha$ is already
shown to be independent of $p \in K[X]$ and $\lambda, \mu\in K$, as it depends
on $p_m$ only, but is still countably infinite at this stage. We show that it
is finite by establishing that, in the unique factorization $\alpha = q^j
\beta_i$ of an occurring $\alpha$, the exponent $j$ is bounded in terms of $d$ and $p_m$ only.

To this end, fix $i$ and consider the set of all $j \geq 0$ (if any) such that
$q^j \beta_i$ occurs in $v$, and arrange them in increasing order, say $0 \leq
j_1 < \ldots < j_t$. We will set out to establish a bound on $j_t$. To start
with, note that $t \leq d$ as $v\in\Ker (p(M)-\mu)$ and $\dim \Ker (p(M) - \mu)
\leq d$. The next step is to obtain a bound on $j_1$. Among all $s$ such that
$(q^{j_1} \beta_i, s) \in \mathcal{O}_v$, let $s_0$ be the largest one. Then
$(q^{j_1} \beta_i, s_0)$ has property NOI. Indeed, if another index in the order
interval
\[[(q^{j_1} \beta_i, s_0), (\frac{q^{j_1}}{q^{m+1}}\beta_i, 1))\]
is in $\mathcal{O}_v$ then, by the choice of $s_0$, such an index must be of
the form $(\frac{q^{j_1}}{q^a}\beta_i, r)$ for some $a > 0$. But then $q^{j_1
-a} \beta_i$ occurs in $v$, contradicting the minimal choice of $j_1$. As
established above, we must have $p_m(\frac{q^{j_1}\beta_i}{q^m})$=0. Hence
$\frac{q^{j_1}\beta_i}{q^m}=q^a\beta_i$ for a uniquely determined $0\leq a\leq
J$ and we conclude that $0\leq j_1\leq J+m$.

The subsequent step is to consider the jumps $j_k-j_{k-1}$ for $k\geq 2$. We
claim that, if $k \geq 2$ and $j_k - j_{k-1} > m$, then $j_k \leq J+m$. To see this, suppose $j_k - j_{k-1} > m$ and among all $s$ such that $(q^{j_k} \beta_i, s) \in \mathcal{O}_v$, let $s_0$ be
the largest one. Then $(q^{j_k} \beta_i, s_0)$ has property NOI. Namely, if
another index in the order interval
\[[(q^{j_k} \beta_i, s_0), (\frac{q^{j_k}}{q^{m+1}}\beta_i, 1))\]
is in $\mathcal{O}_v$, then by the choice of $s_0$, such an index must be of
the form $(\frac{q^{j_k}}{q^a}\beta_i, r)$ for some $0 < a \leq m$. But then
$q^{j_k -a} \beta_i$ occurs in $v$, and $j_{k-1}<j_k-m\leq j_k-a<j_k $. Thus
$j_k-a$ is then properly between $j_{k-1}$ and $j_k$, contradicting the
definition of the $j_l$. We conclude that, if $k\geq 2$ and $j_k-j_{k-1}>m$, we must have $p_m(\frac{q^{j_k}\beta_i}{q^m})=0$ and hence $0\leq j_k\leq J+m$ as above.

Thus, starting at $j_1 \leq J +m$, the jumps $j_k - j_{k-1}$ are at most $m$ as
soon as $j_k \geq J+m+1$ (if ever). Since there are $t-1 \leq d-1$ jumps, and
$j_1 \leq J+m$, we conclude that $j_t \leq J+m +(d-1)m$.

This argument applies to all $\beta_i$ with $i\in\{1,\ldots,h\}$ and hence
there are at most $h(J+m+(d-1)m+1)\leq (\deg (p_m))(J+m+(d-1)m+1)$ possible values
of $\alpha$ occurring in $\mathcal{O}_v$. Since $J$ depends only on $p_m$ the
proof is complete.

Although we will not use this in the sequel, the argument actually shows that for $q\neq 1$ one can take \[\mathcal{L}_{P,d} = \bigoplus_{\substack{i=1,\ldots h,\, 0\leq j\leq J+m +(d-1)m,\, s =1, \ldots, d}} K
\Psi_{q^j\beta_i,s},\]
with the $\beta_1,\ldots,\beta_h$ and $J$ defined as previously in terms of the $\mathbb Z$-action on the non-zero roots of $p_m$ as given by multiplication with powers of $q$.

\end{proof}

\section{Proof of Theorem~\ref{mainthm}}\label{sec:proof}
We can now put the pieces together and prove Theorem~\ref{mainthm}. We may
clearly assume that $K$ is algebraically closed. In the notation of the
theorem, since $P$ is not constant, $\sigma(P)$ is infinite by
Theorem~\ref{eigspec}. If $\lambda_0 \in \sigma(P)$, then $\Ker(P - \lambda_0)$
has finite dimension by Theorem~\ref{thm:cordimker}. Since $P$ and $Q$ commute,
we see that there are infinitely many different pairs $(\lambda_0, \mu_0) \in K
\times K$ in the simultaneous point spectrum with corresponding simultaneous
eigenvectors $v_{\lambda_0, \mu_0}$. As already remarked in Section~\ref{sec:basicnotions},
$\Delta_{P,Q} (M, \lambda_0, \mu_0)v_{\lambda_0, \mu_0} =0$ for all such pairs.
Suppose, then, that $(\lambda_0, \mu_0)$ is in the simultaneous point spectrum
and that $\Delta_{P,Q} (M, \lambda_0, \mu_0)$ is not constant. Then, since
$(P-\lambda_0) v_{\lambda_0, \mu_0} =0$, and the degree of $\Delta_{P,Q} (M,
\lambda_0, \mu_0)$ as a polynomial in $M$ is uniformly bounded by some $d$ as
$\lambda_0$ and $\mu_0$ vary, Theorem~\ref{simeigfin} shows that $v_{\lambda_0,
\mu_0} \in \mathcal{L}_{P,d}$ where $\mathcal{L}_{P,d}$ is a finite dimensional
space which depends only on $P$ and $d$. But by linear independence, this can
happen for at most $\dim\mathcal{L}_{P,d}$ pairs $(\lambda_0, \mu_0)$. For the
remaining infinitely many $(\lambda_0, \mu_0)$, $\Delta(M, \lambda_0, \mu_0)$
must be a constant and then, as $v_{\lambda_0, \mu_0} \neq 0$, it is zero in
$H_K(q)$.

We conclude that, for all $i$,  $\delta_i (\lambda_0, \mu_0) = 0$ for
infinitely many different simultaneous eigenvalues. But then $\delta_i (P,Q)$
has an infinite dimensional kernel and by Theorem~\ref{thm:cordimker} we are
done. \\

\noindent {\bf Acknowledgments.} This work was supported by a visitor's grant
of the Netherlands Organisation for Scientific Research (NWO), the Swedish
Foundation for International Cooperation in Research and Higher Education
(STINT), the Crafoord Foundation, the Royal Physiographic Society in Lund, and
the Royal Swedish Academy of Sciences. We are also grateful to Lars Hellstr\"om
and Daniel Larsson for helpful comments and discussions, and to the referee for his detailed suggestions concerning the exposition of the material.

\end{document}